\documentclass{amsart}

\usepackage{amsmath,amssymb,amsthm,mathrsfs}
\usepackage{tikz}
\usetikzlibrary{positioning}
\usepackage{color}
\usepackage{mathptmx}

\theoremstyle{plain}
\newtheorem{ourtheorem}{Theorem}[section]
\newtheorem{ourcorollary}[ourtheorem]{Corollary}
\newtheorem{ourlemma}[ourtheorem]{Lemma}
\newtheorem{ourproposition}[ourtheorem]{Proposition}
\theoremstyle{definition}

\newtheorem{ourexample}[ourtheorem]{Example}
\newtheorem{ourremark}[ourtheorem]{Remark}

\newcommand{\kernel}{\operatorname{Ker}}

\newcommand{\Dom}{\operatorname{Dom}}
\newcommand{\Rat}{{\mathrm{Rat}}}

\newcommand{\re}{\operatorname{Re}}

\newcommand{\matr}[2]{\ensuremath{\left[\begin{array}{#1}#2\end{array} \right]}}

\newcommand{\al}{\alpha}
\newcommand{\be}{\beta}
\newcommand{\ga}{\gamma}

\newcommand{\la}{\lambda}

\newcommand{\si}{\sigma}

\newcommand{\om}{\omega}\newcommand{\Om}{\Omega}

\newcommand{\BC}{{\mathbb C}}\newcommand{\BD}{{\mathbb D}}

\newcommand{\BN}{{\mathbb N}}
\newcommand{\BP}{{\mathbb P}}

\newcommand{\BT}{{\mathbb T}}

\newcommand{\cP}{{\mathcal P}}

\newcommand{\whatA}{\widehat{A}}\newcommand{\whatB}{\widehat{B}}
\newcommand{\whatC}{\widehat{C}}\newcommand{\whatD}{\widehat{D}}

\newcommand{\ov}[1]{{\overline{#1}}}

\newcommand{\tu}[1]{{\textup{#1}}}
\newcommand{\mat}[1]{\ensuremath{\begin{bmatrix} #1 \end{bmatrix}}}

\newcommand{\sbm}[1]{\left[\begin{smallmatrix} #1\end{smallmatrix}\right]}

\begin{document}

\title[Matrix representation and spectral analysis of unbounded Toeplitz operators]{A Toeplitz-like operator with rational matrix symbol having poles on the unit circle:\ Matrix representation and spectral analysis}

\author[G.J. Groenewald]{G.J. Groenewald}
\address{G.J. Groenewald, School of Mathematical and Statistical Sciences,
North-West University,
Research Focus: Pure and Applied Analytics,
Private~Bag X6001,
Potchefstroom 2520,
South Africa.}
\email{Gilbert.Groenewald@nwu.ac.za}

\author[S. ter Horst]{S. ter Horst}
\address{S. ter Horst, School of Mathematical and Statistical Sciences,
North-West University,
Research Focus: Pure and Applied Analytics,
Private~Bag X6001,
Potchefstroom 2520,
South Africa
and
DSI-NRF Centre of Excellence in Mathematical and Statistical Sciences (CoE-MaSS),
Johannesburg,
South Africa}
\email{Sanne.TerHorst@nwu.ac.za}

\author[J. Jaftha]{J. Jaftha}
\address{J. Jaftha, Numeracy Center, University of Cape Town, Rondebosch 7701, Cape Town, South Africa}
\email{Jacob.Jaftha@uct.ac.za}

\author[A.C.M. Ran]{A.C.M. Ran}
\address{A.C.M. Ran, Department of Mathematics, Faculty of Science, VU Amsterdam, De Boelelaan 1111, 1081 HV Amsterdam, The Netherlands and Research Focus: Pure and Applied Analytics, North-West University, Potchefstroom, South Africa}
\email{a.c.m.ran@vu.nl}

\keywords{Toeplitz operators, unbounded operators, semi-infinite Toeplitz matrices, state space systems, spectrum}

\subjclass{Primary 47B35, 47A53; Secondary 93C05}

\begin{abstract}
In this paper we consider a class of unbounded Toeplitz operators with rational matrix symbols that have poles on the unit circle and employ state space realization techniques from linear systems theory, as used in our earlier analysis in \cite{GtHJR24b} of this class of operators, to study the connection with semi-infinite Toeplitz matrices and to 
determine the essential spectrum and resolvent set.
\end{abstract}

%47B35(1973–now)Toeplitz operators, Hankel operators, Wiener-Hopf operators For other integral operators, see also 45P05, 47G10 [See also 32A25, 32M15]
%47A53(1980–now)(Semi-) Fredholm operators; index theories [See also 58B15, 58J20]
%47A68(1980–now)Factorization theory (including Wiener-Hopf and spectral factorizations) of linear operators
%93C05 Linear systems

\maketitle

\section{Introduction}

In \cite{GtHJR18} a class of unbounded Toeplitz-like operators was introduced, whose symbols are rational functions with poles on the unit circle. The case where the symbol is matrix-valued was first studied in \cite{GtHJR21}. Let $\Omega(z)$ be a rational $m\times m$ matrix function with poles on the unit circle. Then the unbounded Toeplitz-like operator $T_\Om \left (H^p_m \rightarrow H^p_m \right )$ with symbol $\Omega(z)$ from \cite{GtHJR21} is defined  as
\begin{align*}
\Dom(T_\Om)=&\left\{
f\in H^p_m : \Om f = h + r \textrm{ where } h\in L^p_m(\BT),\mbox{ and } r \in\Rat_0^{m}(\BT)\right\}\\
T_\Om f = \BP h, &\textrm{ where } \BP \textrm{ is the Riesz projection of } L^p_m(\BT) \textrm{ onto }H^p_m.
\end{align*}
Here, for $1<p<\infty$, $L^p_m(\BT)$ and $H^p_m$ denote the spaces of vector-valued functions of length $m$ with entries in the Lebesgue space $L^p(\BT)$ over the unit circle $\BT$ and with entries in the Hardy space  $H^p$ over the unit disc $\BD$, respectively, and $\Rat_0^{k}(\BT)$ is the set of strictly proper rational vector-valued functions of length $k$ with poles only on the unit circle.

Assuming that $\det \Om(z)$ is not uniformly zero, in \cite{GtHJR21} a Wiener-Hopf type factorization is determined for $\Om(z)$ leading to a factorization of the Toeplitz-like operator $T_\Om$, from which it can be derived that $T_\Om$ is Fredholm if and only if $\Om(z)$ has no zeroes on $\BT$. Here a zero of $\Om(z)$ corresponds to a pole of $\Om(z)^{-1}$ in which case $\det \Om(z)$ need not be zero, unlike in the case where there are no poles on $\BT$. The Wiener-Hopf type factorization allows one to determine the Fredholm index, in case $\Om(z)$ has no zeroes on $\BT$. However, due to the complicated structure of this Wiener-Hopf type factorization, it appears difficult to determine the dimensions of the kernel and co-kernel of $T_\Om$, and hence whether $T_\Om$ is invertible.  

Using state space representations, in \cite{GtHJR24b} we determine a criterion for invertibility of $T_\Om$ in terms of an associated algebraic Riccati equation. In that setting, we assume that $\Om(z)$ is given by a minimal state space realization of the form  
\begin{equation}\label{Omreal}
\Omega(z)=R_0+zC(I-zA)^{-1}B+\gamma(z I-\alpha)^{-1}\beta,
\end{equation}
where $R_0$, $A,B,C$, and $\al,\be,\ga$ are matrices of appropriate size, and $I$ indicates an identity matrix of appropriate size, such that $A$ has all its eigenvalues in the open unit disc $\BD$ and $\al$ has all its eigenvalues in the closed unit disc $\ov{\BD}$, i.e., $A$ is stable and $\al$ is semi-stable. Note that each rational matrix function $\Om(z)$ admits such a realization. Minimality of the realization \eqref{Omreal} means that there does not exists a realization of the same form with $A$ or $\al$ of smaller size, and implies that the poles of $\Om(z)$ in $\ov{\BD}$ correspond to the eigenvalues of $\alpha$ and the poles in $\BC\backslash \BD$ correspond to the inverses of the eigenvalues of $A$ (to be interpreted as $\infty$ if $0$ is an eigenvalue of $A$). This is equivalent to the triples $(C,A,B)$, and $(\ga,\al,\be)$ both being discrete-time observable and controllable, cf., \cite{HRvS21}, which we will abbreviate by saying that the triples $(C,A,B)$ and $(\ga,\al,\be)$ are minimal. In the present paper we further exploit the state space approach initiated in \cite{GtHJR24b} to determine further properties of the operator $T_\Om$.

As noted in Proposition 2.1 of \cite{GtHJR21}, the space of vector polynomials of size $m$ is contained in the domain of $T_\Om$, and so $T_\Om$ generates a matrix representation with respect to the standard basis in $H^p_m$, which, analogously to the scalar case (see Theorem 1.3 in \cite{GtHJR18}), has a Toeplitz structure with certain conditions on the growth of the coefficients. Using state space realization techniques, in Section \ref{S:ToepMat} we show that conversely a semi-infinite Toeplitz matrix satisfying these conditions necessarily leads to a matrix symbol of the type we consider. We also discuss this Toeplitz matrix structure in the context of the Sarason class of unbounded Toeplitz operators, as proposed in the last section of \cite{S08} and further investigated in \cite{R13,R16}.

In Section \ref{S:SpecAnalysis} we make use of the characterization of Fredholmness and invertibility of $T_\Om$ together with state space techniques to compute the essential spectrum and resolvent set of the operator $T_\Om$. Furthermore, we express these subsets of $\BC$ in terms of the matrices of the realization \eqref{Omreal} and stabilizing solutions to  associated Riccati equations. In the case of continuous rational matrix functions on $\BT$, so that $T_\Om$ is bounded, using a different type of realization Fredholmness and invertibility of $T_\Om$ have been studied in \cite[Chapter XXIV]{GGK93}.

Some examples that illustrate these results are given in Section \ref{S:Example} to conclude the paper. In particular, in one of the examples the essential spectrum turns out to be the full complex plane. This is in contrast to what happens for scalar symbols, in which case the essential spectrum is an algebraic curve in the complex plane.

\section{Semi-infinite Toeplitz matrices}\label{S:ToepMat}

Since the space $\cP^m$ of vector polynomials of size $m$ is contained in the domain of $T_\Om$, for any $m \times m$ rational matrix function $\Om(z)$, $T_\Om$ has a block matrix representation $T_\Om=[t_{i,j}]_{i,j=0}^\infty$ with $t_{i,j}\in \BC^{m \times m}$, with respect to the standard block basis of $H^p$, and this matrix representation has a Toeplitz structure, that is $t_{i,j}=a_{i-j}$ for all $i,j$. The latter follows from the fact that $T_{z^{-1}I}T_\Om T_{z I}f = T_\Om f$ for $f\in\Dom(T_\Om)$.

Conversely, let $[T]=\begin{bmatrix} a_{i-j}\end{bmatrix}_{i,j=0}^\infty$ be a semi-infinite block Toeplitz matrix, where each $a_j$ is an $m\times m$ matrix. In this section we discuss the question:\ Under which conditions on the entries $a_j$ does this semi-infinite block Toeplitz matrix correspond to a block Toeplitz matrix generated by an operator $T_\Om$ as defined above? The results in this section also hold for the case where the entries are non-square matrices; these results can be obtained simply by adding zero-rows or zero-columns to make the matrices square. We consider the following conditions:
%Motivated by Theorem 1.3 in \cite{GtHJR18} we consider the following conditions:
%\footnote{StH: The first condition indeed corresponds to what is in \cite{GtHJR18}, but without rationality it would be more logical to replace it with the assumption that the sequence corresponds to the Taylor coefficients of a function in $H^p_{m \times m}$. In the rational cases, these things all come down to the same thing.}
\begin{equation}\label{eq:conditions}
\begin{aligned}
  \textup{(i)} &\ \  \mbox{$\|a_0\|, \|a_1\|, \ldots$ are the Taylor coefficients of a function in $H^p$; and} \\
  \textup{(ii)} &\ \  \mbox{$\|a_{-j}\|\leq K \binom{j}{M-1}$},\ j\geq 1\mbox{ for some } M\in\mathbb{N},\, K>0.
\end{aligned}
\end{equation}
%\begin{equation}\label{eq:conditions2}
%(\|a_1\|, \|a_2\|, \ldots )\in \ell^2 \textrm{ and } \|a_{-j}\|=O(j^{M-1}), j\geq 1\textrm{ for some } M\in\mathbb{N}.
%\end{equation}
These conditions are not enough to get rationality of the symbol. To achieve that, we shall make use of the theory of state space realizations from systems theory, as outlined in, e.g., \cite[Chapter 3]{HRvS21}. Form the block Hankel matrices
\[
H^+_k =\begin{bmatrix} 
a_1 & a_ 2 & \cdots & a_k \\ a_2 & a_3 & \cdots & a_{k+1} \\
\vdots & \vdots & & \vdots \\
a_k & a_{k+1} & \cdots & a_{2k-1}
\end{bmatrix},
\qquad
H^-_k =\begin{bmatrix} 
a_{-1} & a_ {-2} & \cdots & a_{-k} \\ a_{-2} & a_{-3} & \cdots & a_{-k-1} \\
\vdots & \vdots & & \vdots \\
a_{-k} & a_{-k-1} & \cdots & a_{-2k+1}
\end{bmatrix}.
\]
We shall also impose the condition
\begin{equation}\label{eq:rankHankels}
\max_{k\geq 1} {\rm rank\,} H^+_k =: n_+ <\infty, \qquad \max_{k\geq 1} {\rm rank\,} H^-_k =: n_-
<\infty.
\end{equation}

The following theorem is the main result of this section.

\begin{ourtheorem}\label{T:ToepMat}
A semi-infinite block Toeplitz matrix $T=\begin{bmatrix} a_{i-j}\end{bmatrix}_{i,j=0}^\infty$ with each $a_j$ an $m\times m$ matrix, is the block Toeplitz representation of an operator $T_\Omega$ for a rational $m\times m$ matrix function $\Omega(z)$, possibly with poles on $\BT$, if and only if the conditions \eqref{eq:conditions} and \eqref{eq:rankHankels} hold.
\end{ourtheorem}

\begin{proof}[\bf Proof]
%\proof
Assume that the conditions \eqref{eq:conditions} and \eqref{eq:rankHankels} hold. Then by the realization algorithm (cf., Section 3.4 in \cite{HRvS21}) there is a minimal triple of matrices $(A,B,C)$ such that $a_j=CA^{j-1}B$ for  $j\geq 1$, and there is a minimal triple of matrices $(\alpha,\beta,\gamma)$ such that $a_{-j}=\gamma\alpha^{j-1}\beta$ for $j\geq 1$. Then for $|z|$ small enough we have
\[
\Om_+(z):=\sum_{j=1}^\infty a_j z^j= zC(I-zA)^{-1}B,
\]
while for $|z|$ large enough we have
\[
\Om_-(z):=\sum_{j=1}^\infty a_{-j}z^{-j} =\gamma(z I-\alpha)^{-1}\beta.
\]

By condition (i) in \eqref{eq:conditions} it follows that each entry in the matrix function $\Om_+(z)$ is in $H^p$, and hence $\Om_+(z)\in H^p_{m\times m}$, which implies that $\Om_+(z)$ has no poles in the closed unit disc $\overline{\BD}$. Then the minimality of the triple $(A,B,C)$ implies that $A$ has all its eigenvalues inside $\BD$, and thus $A$ is stable (cf., Corollary 8.14 in \cite{BGKR08}). 

According to (ii) in \eqref{eq:conditions}, there exist $M\in\BN$ and $K>0$ such that  $\|a_{-j}\|\leq K \binom{j}{M-1}$ for all $j\geq 0$. This implies that for each $r>1$ the function $\Om_-(z)$ is analytic on $|z| >r$, since $\sum_{j=1}^\infty K \binom{j}{M-1} r^{-j} $ is a convergent series (for instance by the ratio test). Hence $\Om_-(z)$ is analytic on $\BC \setminus \overline{\BD}$, so that $\al$ has al its eigenvalues inside $\overline{\BD}$ and hence is semi-stable, now using minimality of $(\al,\be,\ga)$ again in combination with  Corollary 8.14 in \cite{BGKR08}.

Thus, the function 
\[
\Omega(z)=a_0+zC(I-zA)^{-1}B+\gamma(z I-\alpha)^{-1}\beta
\] 
is a rational matrix function which possibly has poles on the unit circle, and for this $\Omega(z)$ we claim that $T$ coincides with the Toeplitz block matrix representation of $T_\Om$. To see this, we consider the expansion of $\Om(z)z^n$ for $n=0,1,\ldots$, as in the proof of Lemma 4.1 of \cite{GtHJR24b}. By successively using the identity
\[
(zI-\al)^{-1}=z^{-1}I+z^{-1}\al(zI-\al)^{-1},
\]
it follows that for each $w\in\BC^m$ we have 
\begin{align*}
\Om(z)z^n w &= z^{n+1}C(I-zA)^{-1}Bw + a_0 z^nw + \sum_{k=0}^{n-1}z^{n-k-1}\ga \al^k\be w   + \ga \al^n (zI-\al)^{-1}\be w \\
 &= \sum_{j=0}^\infty z^{n+1+j} CA^jBw + a_0 z^nw + \sum_{k=0}^{n-1}z^{n-k-1}a_{-k-1}w  + \ga \al^n (zI-\al)^{-1}\be w\\
 &= z^n\left(\sum_{j=0}^\infty z^{j+1} a_{j+1} + a_0  + \sum_{l=-n}^{-1}z^{l}a_{l}\right)w  + \ga \al^n (zI-\al)^{-1}\be w\\
  &= \sum_{j=0}^\infty a_{j-n} w z^j  + \ga \al^n (zI-\al)^{-1}\be w.
\end{align*}
Due to the eigenvalues of $\al$ being in $\ov{\BD}$, it follows that $\ga \al^n (zI-\al)^{-1}\be w$ can be written as a function in $\Rat_0^{m}(\BT)$ and an anti-analytic part in $L^p_{m}(\BT)$, showing that the $T_\Om z^n w$ is equal to $\sum_{j=0}^\infty a_{j-n} w z^j$, which confirms the claim regarding the Toeplitz block matrix representation of $T_\Om$.

For the converse implication, assume $\Om(z)$ is an $m \times m$ rational matrix function, which may have poles on the unit circle. An argument similar to the analogous result for the scalar case in \cite[Theorem 1.3]{GtHJR18} works here as well, but we give an alternative argument using state space realizations. Note that $\Om(z)$ admits a minimal state space realization as in \eqref{Omreal} with $A$ stable and $\al$ semi-stable. In the previous paragraph we showed that $T_\Om$ has a block Toeplitz matrix representation $[a_{i-j}]_{i,j=0}^{\infty}$ with $a_0=R_0$, $a_j=CA^{j-1}B$, for $j\geq 1$, and $a_{j}=\ga \al^{-j-1}\be$ for $j\leq -1$. Then $H_k^+$ and $H_k^-$ are the $k \times k$ block Hankel matrices associated with the rational matrix functions $\Om_+(z)$ and $\Om_-(z)$, as defined above, respectively, and hence \eqref{eq:rankHankels} holds by the realization algorithm \cite[Section 3.4]{HRvS21}. Since $A$ is stable, it is clear that 
\[
\sum_{j=1}^\infty z^j a_j= \sum_{j=1}^\infty z^j CA^{j-1}B=zC(I-zA)^{-1}B\ \mbox{is in $H^p_{m \times m}$}.
\]
Hence item (i) in \eqref{eq:conditions} holds. 

Since $a_{-j}=\gamma \alpha^{j-1} \beta$ for $j\in\BN$, we have $\|a_{-j}\|\leq \|\ga\| \|\be\| \|\al^{j-1}\|$.
To get an upper bound on the growth of $\|\al^{j-1}\|$ let $\alpha=SJS^{-1}$ where $J$ is in Jordan normal form. Then $\|\alpha^{j-1}\| \leq \|S\|\cdot \|S^{-1}\| \|J^{j-1}\|$. Let $M$ be the size of the largest Jordan block in $J$ corresponding to an eigenvalue $\lambda_0$ of $\alpha$ on the unit circle, and let $\lambda_0 I+N$ be that Jordan block. Then  $\|J^{j-1}\| \leq \|(\lambda_0 I +N)^{j-1}\|$. Moreover, using the fact that the norm of an $M\times M$ matrix is bounded by $M$ times the maximum absolute value of the entries in the matrix, we see that for some constant $K>0$ we have
\[
\|J^{j-1}\| \leq K \cdot \begin{pmatrix} j \\ M-1 \end{pmatrix},
\]
as the largest absolute value of an entry in $(\lambda_0 I+N)^{j-1}$ occurs in the right upper corner of the matrix. This proves the theorem.
% which may increase the constant $K$, and need only consider the Jordan blocks with eigenvalue on $\BT$, and of those it is only the size of the largest Jordan block, denoted by $M$ matters. Hence, we need to find an upper bound on $\|J_M(1)^j\|$, where $J_M(1)$ is the Jordan block of size $M$ with eigenvalue $1$. I DON'T THINK WE GET THIS BY TAKING THE LARGEST NUMBER, BUT PROBABLY THIS IS KNOWN. 
%
%ARGUMENT BY ANDRE, SHOULD STILL REWRITE:
%the $a_{-j}=\gamma \alpha^{j-1} \beta$. Now
%$\alpha$ is semi-stable. Considering the Jordan normal form of $\alpha$,
%it is easily seen that $\| a_{-j}\| \leq K\cdot  j^{M-1}$ where $M$ is the size
%of the largest Jordan block of $\alpha$ with an eigenvalue on the unit circle,
%and $K$ is some constant (that can be seen by considering the powers
%of $I +N$ where $N$ is an $M\times M$ upper triangular with ones on the first
%diagonal above the main diagonal. I think that $j^{M-1}$ is a gross overestimate
%for M large and j large, but certainly an upper bound that will work.
%
%SECOND ARGUMENT WITH UPDATED BOUND.
%\[
%\| \alpha^{j-1}\| \leq K\cdot \begin{pmatrix} j \\ M-1 \end{pmatrix}.
%\]
%
%The number in the right upper corner of $(I+N)^{j-1}$ is $j$ choose $M-1$
%when $N$ is $M \times M$. To give an idea, in the jpeg file it is done for $M$ is 6
%and $j$=1 to 30: the norms of $(I+N)^{j-1}$ in blue, in red $j$ choose $M-1$.
%Looks like same growth to me.
%And this can obviously be bounded by $j^{M-1}$ times a constant.
%\qed
\end{proof}

While the theorem only claims existence of an $m \times m$ rational matrix function $\Om(z)$ such that the given semi-infinite block Toeplitz $T$ corresponds to the block Toeplitz representation of $T_\Om$, the realization algorithm of Section 3.4 in \cite{HRvS21} used in the proof actually gives a construction of the matrices $A,B,C$ and $\al,\be,\ga$ that together with $R_0:=a_0$ form a minimal state space realization of $\Om(z)$ once it is established that \eqref{eq:rankHankels} holds.  

\begin{ourremark}
We connect the result of Theorem \ref{T:ToepMat} with the {\em Sarason class of unbounded Toeplitz operators}, as proposed by Sarason in the last section of \cite{S08}, and investigated further by Rosenfeld in \cite{R13,R16} (who referred to such operators as Sarason-Toeplitz operators). The idea stems from the observation that a bounded operator $T$ on $H^2$ is a Toeplitz operator with symbol from $L^\infty$ if and only if $ S^* T S=T$, with $S=T_z$ the forward shift operator on $H^2$, i.e., the Toeplitz operator with symbol $z\mapsto z$. To remain in the context of the present paper, we extend the definition of \cite{R13,R16}, and say that an operator $T(H^p_m\to H^p_m)$ is in the  {\em Sarason class of unbounded Toeplitz operators} (abbreviated to {\em Sarason class} in the sequel) whenever     
\begin{itemize}
 \item[(i)] $T$ is closed and densely defined;
  
\item[(ii)] $\Dom(T)$ is $T_{zI_m}$-invariant;
  
\item[(iii)] $T_{z^{-1}I_m} T T_{zI_m}= T$; 
 
\item[(iv)] If $f\in \Dom(T)$ and $f(0)=0$, then $T_{z^{-1}I_m} f\in \Dom(T)$.  
\end{itemize}  
Here $T_{zI_m}$ and $T_{z^{-1}I_m}$ are the bounded Toeplitz operators with symbols $z\mapsto zI_m$ and $z\mapsto z^{-1}I_m$. For the case $m=1$ and $p=2$ it was shown in \cite{GtHJR19b} that the unbounded Toeplitz operators considered in this paper fall within the Sarason class, and it is not difficult to see that this extends to $1<p<\infty$ (as already noted in \cite{GtHJR19b}) and to arbitrary $m\in\BN$ (for items (i)--(iii) see Proposition 2.1 in \cite{GtHJR21}).

Let us consider an operator $T$ in the Sarason class with the property that the constant functions are in $\Dom(T)$. Since $\Dom(T)$ is $T_{zI_m}$-invariant by (ii), it follows that the vector polynomials of size $m$ are in $\Dom(T)$. Hence $T$ admits a block matrix representation with respect to the standard block basis of $H_m^p$, and it is easy to see that (iii) implies that the block matrix representation has Toeplitz structure, i.e., $[T]=[a_{i-j}]_{i,j=0}^\infty$ for matrices $a_k\in\BC^{m \times m}$. 

In \cite{GtHJR19b} we show that the operator in the Sarason class is not uniquely determined by the symbol, since for a scalar-valued rational function $\omega$ the adjoint $T_{\om}^*$ of the unbounded Toeplitz operator $T_\omega$
is a restriction of the unbounded Toeplitz operator $T_{\omega^*}$, and thus the symbol $\om^*$ defines two different operators in the Sarason class, with the same symbol.
% (assuming the Sarason sub-symbols are indeed what we thing that the symbols should be). 

In short, for an operator in the Sarason class:
%HERE IS WHAT WE CAN PROBABLY EXPECT, AND SOME THINGS TO INVESTIGATE FURTHER:
\begin{itemize}
  \item[(i)] We get a matrix representation with respect to the standard basis, and this matrix representation has a Toeplitz structure.  
  
  \item[(ii)] The Toeplitz matrix structure generates an abstract Fourier representation, and the fact that $\cP^m\subset \Dom(T)$ implies that the 'analytic part' must be in $H^p_m$. Apart from that, we are not sure if there are any conditions on the matrix entries that can be obtained.
 \end{itemize}
From that perspective, we can state the following \emph{open problem}.  Start with a Toeplitz matrix and assume the 'analytic part' is in $H^p_m$. Can we then define an unbounded Toeplitz operator from that? It is clear how $T$ acts on the vector polynomials of size $m$. Is there maybe a unique way to extend the operator on $\cP^m$ to one that satisfies (i)--(iv)? Even without uniqueness this would be a relevant result.        
\end{ourremark}

\section{Spectral analysis of $T_\Omega$}\label{S:SpecAnalysis}

In this section we use state space methods to determine some properties of the spectrum of $T_\Om$. In the case that $\Om(z)$ is a scalar function, the spectrum was described in detail in \cite{GtHJR19a}. For matrix-valued symbols the situation is much more complicated, as mentioned in the introduction. In our results, the starting point is that $\Om(z)$ is given by a state space realization of the form \eqref{Omreal}
with $A$ stable and $\alpha$ semi-stable. Throughout this section, for a square matrix $M$ we write $\si(M)$ for the spectrum of $M$, i.e., the set of its eigenvalues, and $\rho(M):=\BC \backslash \si(M)$ for the resolvent set of $M$.

\subsection{The essential spectrum}

Recall that the essential spectrum of $T_\Om$ is defined by 
\[
\si_{\textup{ess}}(T_\Om)=\left\{\la\in\BC \colon T_\Om -\la I_{H^p} \mbox{ is not Fredholm} \right\}.
\]
We can write $T_\Om -\la I_{H^p}=T_{\Om -\la I_m}$ and it then follows from Theorem 1.3 in \cite{GtHJR21} that $\la$ is in $\si_{\textup{ess}}(T_\Om)$ precisely when $\Om(z) -\la I_m$ has a zero on $\BT$, that is, $(\Om(z) -\la I_m)^{-1}$ has a pole on $\BT$, provided that $\det(\Om(z) -\la I_m)$ is not uniformly zero (in $z$ for the fixed choice of $\lambda$). In particular, the essential spectrum of $T_\Om$ does not depend on $p$. A complication that occurs in the case of matrix symbols with poles on $\BT$ is that the symbol can have a pole and a zero at the same point of $\BT$, so that the zeroes of $\Om(z) -\la I_m$ on $\BT$ do not necessarily correspond to the zeroes of $\det(\Om(z) -\la I_m)$ on $\BT$. 

\begin{ourtheorem}\label{thm:essspec}
Let $\Om(z)$ be a rational $m\times m$ matrix function given by a minimal state space realization \eqref{Omreal} with $A$ stable and $\al$ semi-stable.
Define the linear pencil
\[
L(\la,z):=\mat{zA- I&0&B\\
0 & \al-  z I&\be\\ 
zC&\ga&R_0-\la I},\quad \la, z\in\BC.
\]
Then we have the following two descriptions of $\sigma_{\textup{ess}}(T_\Om)$:
\[
\si_{\textup{ess}}(T_\Om)= \{\la\in\BC \colon \det L(\la,\nu)=0 \mbox{ for some } \nu\in \BT\},
\]
and
\begin{align*}
\si_{\textup{ess}}(T_\Om)&= \{\la\in\BC \colon \lambda \in \sigma(\Omega(\nu)) \mbox{ for some }\nu\in \BT\setminus\sigma(\alpha)\}\cup
\\
&\qquad\qquad \cup 
\{\lambda\in\BC \colon  \det L(\la,\nu)=0 \mbox{ for some } \nu\in\mathbb{T}\cap \sigma(\alpha)\}.
\end{align*}
\end{ourtheorem}

First we deal with $\la\in\BC$ for which $\det(\Om(z) -\la I_m)$ is uniformly zero, i.e.,  $\la$ in the set
\begin{equation}\label{EOm}
E(\Omega)=\{ \lambda\in\BC \colon \det(\Om(z)-\lambda I_m)=0 \mbox{ for all } z\in\mathbb{C} \mbox{ that are not poles of } \Om(z) \}.
\end{equation}
The following example shows that this set can be nonempty.

\begin{ourexample}\label{Ex1}
Consider 
\[
\Om(z)=\begin{bmatrix} 2 & \tfrac{1}{z-1} \\ 0 & 2\end{bmatrix},\quad z\in\BD,
\] 
for which $\Om(z)-\la I_2$ has both a pole and zero at $z=1$, for each $\la\in\BC$. However, we have that $\det(\Om(z) -\la I_2)=(2-\la)^2$ is constant in $z$, and in particular, $\det(\Om(z) -2 I)=0$ for all $z$. Clearly, this only happens for $\la=2$. Hence $E(\Om)=\{2\}$.
\end{ourexample}

\begin{ourlemma}\label{L:EOm}
For any $m \times m$ rational matrix function $\Om(z)$ the set $E(\Om)$ consists of at most $m$ points and 
\[
E(\Om)\subset \si_p(T_\Om)\cap \si_{\textup{ess}}(T_\Om).
\]
\end{ourlemma}

\begin{proof}[\bf Proof]
Suppose that $\la\in E(\Om)$. Write $\Om(z)=\frac{1}{q(z)} P(z)$, where $q(z)$ is the least common multiple of all denominators of entries of $\Om(z)$ and $P(z)$ is an $n\times n$ matrix polynomial. Then 
\[
\det(\Om(z)-\la I_m)=\frac{1}{q(z)^m} \det (P(z)-\la q(z)I_m).
\]
This is zero for all $z$ if and only if $\det(P(z)-\la q(z) I_m)=0$ for all $z$. The latter expression is an $m$-th degree polynomial in $\la$ with coefficients which are polynomials in $z$. For $z$ not a zero of $q$ there are precisely $m$ solutions $\la$, counting multiplicities. This means that $E(\Om)$ consists of at most $m$ points. 

According to Forney \cite{F75}, when viewing $\Om(z)-\la I_m$ as a matrix over the field of rational functions in $z$, for $\la\in E(\Om)$ there is a basis for $\kernel(\Om(z)-\la I_m)$ consisting of vector polynomials. Now, if $(\Om(z)-\la I_m)r(z)$ is uniformly $0$ for a nontrivial vector polynomial $r(z)$,  then $r(z)\in \kernel (T_\Om-\la I)$, so that $\la\in\si_p(T_\Om)$. However, also $(\Om(z)-\la I_m)r(z)h(z)$ is uniformly zero for any $h\in H^p$. Hence $rH^p\in \kernel (T_\Om-\la I)$, so that $\kernel (T_\Om-\la I)$ is infinite dimensional and thus $\la\in\si_{\textup{ess}}(T_\Om)$. Hence $E(\Om)\subset \si_p(T_\Om)\cap \si_{\textup{ess}}(T_\Om)$.
\end{proof}

We now prove Theorem \ref{thm:essspec}.

\begin{proof}[\bf Proof of Theorem \ref{thm:essspec}]
Decompose $L(\la,z)$ as 
\[
L(\la,z)=\matr{cc|c}{zA- I&0&B\\
0 & \al-  z I&\be\\ 
\hline
zC&\ga&R_0-\la I}=\mat{\whatA(z)&\whatB\\\whatC(z)&\whatD_{\la}},\quad \la, z\in\BC.
\]
Note that $\whatA(z)$ is invertible if and only if $z$ is not a pole of $\Om(z)$, due to the minimality of the realization \eqref{Omreal}. In that case, we have that 
\begin{align*}
\Om(z)-\la I 
&= R_0 - \la I+\begin{bmatrix} zC &\gamma\end{bmatrix}
\begin{bmatrix} I-zA & 0 \\ 0 & zI-\alpha\end{bmatrix}^{-1}
\begin{bmatrix} B\\ \beta\end{bmatrix}\\
&= R_0 - \la I-\begin{bmatrix} zC &\gamma\end{bmatrix}
\begin{bmatrix} zA -I & 0 \\ 0 &\alpha- zI\end{bmatrix}^{-1}
\begin{bmatrix} B\\ \beta\end{bmatrix} 
= \whatD_{\la}-\whatC(z)\whatA(z)^{-1}\whatB.
\end{align*}
Hence $\Om(z)-\la I$ is the Schur complement of $L(\la,z)$ with respect to $\whatA(z)$; see, e.g., \cite[Section 2.2]{BGKR08}. It follows that, whenever $\det \whatA(z)\neq 0$, $\Om(z)-\la I$ is invertible if and only if $L(\la,z)$ is invertible.

Now take $\la_0\in E(\Om)$, with $E(\Om)$ as in \eqref{EOm}. Then for any $\nu\in\BT$ with $\det \whatA(\nu)\neq 0$, we have that $\det(\Om(\nu)-\la_0 I)=0$ (so that $\la_0\in\si(\Om(\nu))$) and hence $\det L(\la_0,\nu)=0$, while also $\la_0 \in \si_{\textup{ess}(T_{\Om})}$, by Lemma \ref{L:EOm}. So for $\la_0\in E(\Om)$, the statements hold.

From now on we assume $\la_0\notin E(\Om)$. Applying the standard Schur complement inversion formula \cite[Section 2.2]{BGKR08} to $L(\la_0,z)$  with respect to $\whatA(z)$, we have that
\begin{equation}\label{InvForm}
(\Om (z)- \lambda_0 I)^{-1} = E L(\lambda_0, z)^{-1} E^*, \mbox{ with $E=\mat{0&0&I}$},\quad \mbox{whenever } \det L(\la_0,z)\neq0.
\end{equation}

%\begin{equation}\label{InvForm}
%(\Om (z)- \lambda_0 I)^{-1} = \begin{bmatrix} 0 & 0 & I \end{bmatrix} L(\lambda_0, z)^{-1}
%\begin{bmatrix} 0 \\ 0 \\ I \end{bmatrix},\quad \mbox{whenever } \det L(\la_0,z)\neq0.
%\end{equation}

Note that the above formula for $(\Om (z)- \lambda_0 I)^{-1}$ extends to the poles of $\Om(z)$, as long as $L(\la_0,z)$ is invertible. Recall that the zeroes of $\Om (z)- \lambda_0 I$ are the poles of $(\Om (z)- \lambda_0 I)^{-1}$, and that $T_{\Om} -\la_0 I$ is not Fredholm if and only if $\Om (z)- \lambda_0 I$ has a zero on $\BT$, by Theorem 1.3 in \cite{GtHJR21}. Hence, to prove the first formula for $\si_{\textup{ess}}(T_\Om)$ it suffices to show that the poles of $(\Om (z)- \lambda_0 I)^{-1}$ coincide with the zeroes of $\det L(\la_0,z)$, because $L(\la_0,z)$ is a matrix polynomial in $z$.
Define 
\[
N(\la_0,z)=\mat{L(\la_0,z) & E^*\\ E &0}\mbox{ with $E$ as in \eqref{InvForm}}.
\]
By \eqref{InvForm}, $(\Om (z)- \lambda_0 I)^{-1}$ is the Schur complement of $N(\la_0,z)$ with respect to $L(\la_0,z)$. Because the pairs $(A,B)$ and $(\alpha,\beta)$ are controllable, $A$ is stable and $\alpha$ is semi-stable, the matrix
\[
\begin{bmatrix}
zA- I & 0 & B \\ 0 & z I -\alpha & \beta
\end{bmatrix}
\]
has full row rank for all $z\in\mathbb{C}$. This implies that $\mat{L(\la_0,z) & E^*}$ also has full row rank for all $z\in\mathbb{C}$. From this it follows that the Smith form of $\mat{L(\la_0,z) & E^*}$ is $\mat{I\  & 0}$, so that, by Proposition 2.1 in \cite{DNZ} the matrix polynomials $L(\la_0,z)$ and $E^*$ are right coprime. 
Because the pairs $(C,A)$ and $(\gamma, \alpha)$ are observable, $A$ is stable and $\alpha$ is semi-stable, it can be seen in a similar manner that $L(\lambda_0, z)$ and $E$ are left coprime. It then follows from a result of Rosenbrock \cite{R70} (see also Theorem 3.5 in \cite{DNZ} for a nice presentation and extension of this material), that $L(\lambda_0, z)^{-1}$ and $(\Om(z)-\lambda_0)^{-1}$ have the same poles, which, as was explained above, coincide with the zeroes of $\det L(\la_0,z)$. This completes the proof of the first description of the essential spectrum.

For the second description, we may still assume $\la_0\notin E(\Om)$. Let $\nu\in\BT\setminus \si(\al)$ By Schur's determinant formula, applied to $L(\la_0,\nu)$, we have 
\[
\det L(\la_0,\nu)=\det (\Om(\nu)-\lambda_0) \det \whatA(\nu)= \det (\Om(\nu)-\lambda_0I) \det(\nu A-I) \det (\nu I-\alpha).
\] 
Hence $\det L(\la_0,\nu)=0$ if and only if $\la_0$ is an eigenvalue of $\Om(\nu)$. That proves the second description of the essential spectrum.
\end{proof}

\subsection{The resolvent set}

Next, we consider the resolvent set. 
Using results from \cite{GtHJR24b}, see also \cite{FKR10}, we can characterize invertibiliy of $T_\Omega -\lambda I=T_{\Omega -\lambda I}$ in terms of existence of the stabilizing solution to an algebraic Riccati equation. That gives a result that makes it possible to determine whether or not $\lambda$ is in the resolvent set of $T_\Omega$. To be precise, we have the following proposition.

\begin{ourproposition}\label{prop:resolvent}
Let $\Omega$ be given by the realization \eqref{Omreal}. Then $\lambda\in\rho(T_\Omega)$ 
 if and only if there is a matrix $Q$, depending on $\lambda$, such that $R_0-\gamma QB-\lambda I$ is invertible, $Q$
satisfies the algebraic Riccati equation
\begin{equation}\label{eq:Ricc}
Q=\alpha QA+(\beta -\alpha QB)(R_0-\gamma QB-\lambda I)^{-1}(C-\gamma QA),
\end{equation}
and such that the matrices
\[
A_{\circ}:=A-B(R_0-\gamma QB-\lambda I)^{-1}(C-\gamma QA),\ \
\alpha_{\circ}:=\alpha -\left(\beta-\alpha QB\right)(R_0-\gamma QB-\lambda I)^{-1}\gamma
\]
are both stable.
\end{ourproposition}

\begin{proof}[\bf Proof]
By Lemma \ref{L:EOm} we have that $E(\Om)\subset \si_{\textup{ess}}(T_\Om)$. Hence we can restrict to $\la\notin E(\Om)$, i.e., $\det(\Om(z)-\la I)$ is not uniformly zero. In that case, the main result of \cite{GtHJR24b} can be applied, and it follows that the operator $T_\Omega-\lambda I=T_{\Om-\la I_m}$ is invertible if and only if there is a solution $Q$ to the algebraic Riccati equation \eqref{eq:Ricc} for which $A_\circ$ and $\alpha_\circ$ are stable. 
\end{proof}

We have the following corollary, which concerns a special case.

\begin{ourcorollary}\label{C:polesindisc}
Let $\Omega$ be given by the realization \eqref{Omreal} and suppose that $\Omega(z)$ has all its poles in $\ov{\BD}$, and the limit at infinity of $\Omega(z)$ exists. Then $\lambda\in\sigma(T_\Omega)$ if and only if $\lambda\in\sigma(R_0)$ or there is a $\mu\notin \mathbb{D}$, $\mu$ not a pole of $\Om(z)$,  for which $\lambda\in\sigma(\Omega(\mu))$ or  $\det K(\la,\mu)=0$ for $\mu\in\BT$ a pole of $\Om(z)$, where 
\[
K(\la,z)=\mat{\alpha - zI & \beta \\ \gamma & R_0-\lambda I},\quad \la,z\in\BC.
\]
\end{ourcorollary}

\begin{proof}[\bf Proof] 
Under the assumption that $\Omega(z)$ has all its poles inside $\overline{\BD}$, and the limit at infinity of $\Omega(z)$ exists, the matrix $A$ is zero on the zero-dimensional space. In that case the Riccati equation is vacuous, and it follows from Proposition \ref{prop:resolvent} that $\la\in\BC$ is in the resolvent of $T_\Omega$ if and only if  $R_0-\lambda I$ is invertible (i.e., $\la\in \rho(R_0)$) and $\al_\circ:= \alpha -\beta (R_0-\lambda I)^{-1} \gamma$ is stable. 

Note that due to the observations above, the matrix polynomial $K(\la,z)$ coincides with the matrix polynomial $L(\la,z)$ in Theorem \ref{thm:essspec}. For $\la\in \rho(R_0)$, the Schur complement of $K(\la,z)$ with respect to $R_0-\lambda I$ is equal to $\al_\circ-z I$. Hence in that case $\al_\circ- zI$ is invertible if and only if $K(\la,z)$ is invertible. Since $\al_\circ$ is stable precisely when $\al_\circ - \mu I$ is invertible for each $\mu\notin\BD$, it follows that $\la$ is in the resolvent of $T_\Om$ if and only if $\la\in \rho(R_0)$ and $\det K(\la,\mu)\neq 0$ for all $\mu\notin\BD$. Stated in the negative, $\la$ is in $\si(T_\Om)$ if and only if $\la\in \si(R_0)$ or $\det K(\la,\mu)=0$ for some $\mu\notin\BD$. 

To complete the proof, note that for $\mu\notin\BD$ which is not a pole of $\Om(z)$, i.e., such that $\mu\in\rho(\al)$, the Schur complement of $K(\la,\mu)$ with respect to $\al - \mu I$ is equal to 
\[
R_0-\lambda I + \gamma (\mu I -\alpha)^{-1}\beta =\Omega(\mu)-\lambda I.
\]
Hence, in that case $\det K(\la,\mu)=0$ holds if and only if $\det(\Omega(\mu)-\lambda I)=0$, i.e., $\la\in\si(\Om(\mu))$. 
\end{proof}

\section{Examples}\label{S:Example}

We conclude the paper with a number of examples. The first three examples are of scalar symbols, which allows us to compare with the results from \cite{GtHJR19a}, where the spectrum of $T_\omega$ for a scalar symbol $\omega$ is described in detail. The fourth example gives a $2\times 2$ matrix example, illustrating the methods outlined here for a fairly simple case. The final example illustrates the methods for a case in which $E(\Omega)$ is not empty.

\begin{ourexample}
We start with an easy example. Let $\omega(z)=\frac{z-(a+ib)}{z-1}=1+\frac{1-(a+ib)}{z-1}$. This symbol is also considered in Example 4.1 in \cite{GtHJR19a}. For the matrices in the realization we can take $R_0=1$, $A=0$ on the zero-dimensional space, $B=0$, $C=0$, $\alpha=1$ on a one-dimensional space, $\gamma=1$, $\beta=1-a-bi$. We are in the setting of Corollary \ref{C:polesindisc}. The algebraic Riccati equation has the solution $q=0$ as a map from a zero-dimensional space to a one-dimensional space. The matrix $A_\circ$ is also zero on the zero-dimensional space, while $\alpha_\circ= 1-(1-a-bi)(1-\lambda)^{-1}$, for which we require $\la\neq 1=R_0$. The only further condition on invertibility is then that $|\alpha_\circ| <1$. So, $\lambda\in \sigma(T_\omega)$ if and only if $|\alpha_\circ|\geq 1$. In turn this is equivalent to $|a+bi-\lambda | \geq |1-\lambda|$. Note that this inequality is always satisfied for $\la=1$. For $\lambda=x+iy$ we have $|a+bi-\lambda | = |1-\lambda|$ if and only if $2by=a^2+b^2-1 +(2-2a)x$. This is precisely the line $L$ given in Example 4.1 in \cite{GtHJR19a}. For this example we retrieve the results as presented there.
\end{ourexample}

%\bigskip
\begin{ourexample}
Let $\omega(z)=\frac{z^2+1}{z-1}=z+1+\frac{2}{z-1}$.
The matrices in the realization can be taken to be $R_0=1, A=0, C=B=1, \alpha =1, \beta=2, \gamma=1$. Here $A$ and $\alpha$ are $1\times 1$ matrices. The algebraic Riccati equation for $T_\omega -\lambda I$ becomes
\[
q=(2-q)(1-q-\lambda)^{-1}.
\]
Further, $A_\circ=-(1-q-\lambda)^{-1}$ and $\alpha_\circ= 1-(2-q)(1-q-\lambda)^{-1}$. The algebraic Ricatti equation gives $\lambda=2-q-\frac{2}{q}$. Hence $1-q-\lambda=-1+\frac{2}{q}=\frac{2-q}{q}$. If follows that we should have $q\neq 2$ and $q\neq0$, but it is easily checked that $q=2$ and $q=0$ do not satisfy the algebraic Ricatti equation, irrespectively of the value of $\la$. We further get that $A_\circ=\frac{q}{q-2}$ and $\al_\circ=1-q$. Stability of $A_\circ$ and $\al_\circ$ just means $|A_\circ|<1$ and $|\al_\circ|<1$, which can be expressed as $\re q <1$ and $|1-q|<1$, respectively.

So, $\lambda$ is in the resolvent set of $T_\omega$ precisely when $\lambda=2-q-\frac{2}{q}$ for a $q$ that is in the open left half of the disc of radius $1$ with centre at $1$. 
The semi-circular part of the boundary of that set maps precisely to part of the essential spectrum, namely the part of the curve $z=x+iy$ parametrized by
$x=\cos(\theta), y=-\frac{\sin(\theta)\cos(\theta)}{1-\cos(\theta)}$ for $-\pi/2 <\theta < \pi/2$.
(Recall from \cite{GtHJR19a} that the essential spectrum of $T_\omega$ is given by the curve $\omega(\mathbb{T})$.)
The spectrum is the set to the left of that, the essential spectrum has an extra "loop" parametrized by this curve for $\pi/2<\theta<3 \pi/2$, see Figure 1.

\begin{figure}
\includegraphics[height=5cm]{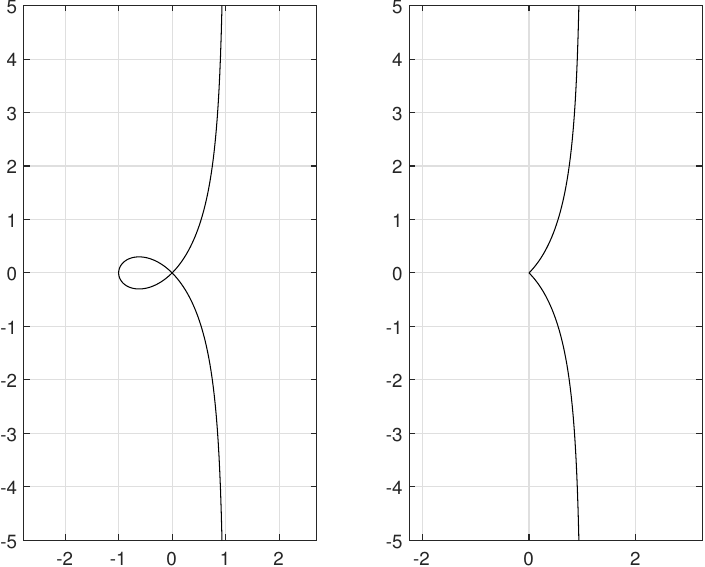}
\caption{
Left: the essential spectrum. Right: the boundary of the spectrum.
The spectrum is the set left of the boundary line.}
\end{figure}
\end{ourexample}

%\bigskip
\begin{ourexample}
Let $\omega(z)=\tfrac{z^3+3z+1}{z^2-1}=z+\tfrac{5}{2}\cdot \tfrac{1}{z-1}+\tfrac{3}{2}\cdot \tfrac{1}{z+1}$. (Compare Example 5.2 in \cite{GtHJR19a}.) The matrices in a realization can be taken to be
\[
R_0=0,\ A=0,\ B=C=1 \mbox{ (all $1\times 1$), and } \alpha=\begin{bmatrix} 1 & 0 \\ 0 & -1\end{bmatrix},\ \beta=\begin{bmatrix} 1 \\ 1 \end{bmatrix},\ \gamma=\begin{bmatrix} \tfrac{5}{2} & \tfrac{3}{2}\end{bmatrix}. 
\]
The algebraic Riccati equation for $T_{\omega -\lambda}$ becomes, with $q=\sbm{q_1\\q_2}$ a $2\times 1$ matrix:
\[
\begin{bmatrix} q_1 \\ q_2\end{bmatrix} =-\begin{bmatrix}1- q_1 \\ 1+ q_2\end{bmatrix}
( \tfrac{5}{2} q_1+\tfrac{3}{2}q_2+\lambda)^{-1}.
\]
Using this equation we can express $\lambda$ in $q_1$ and $q_2$. We obtain
\begin{align*}
\lambda& =-( \tfrac{5}{2} q_1+\tfrac{3}{2}q_2)+\tfrac{-1+q_1}{q_1} ,\\
\lambda& =-( \tfrac{5}{2} q_1+\tfrac{3}{2}q_2)+\tfrac{-1-q_2}{q_2} .
\end{align*}
This also gives a relation between $q_1$ and $q_2$: $\tfrac{-1+q_1}{q_1}=\tfrac{-1-q_2}{q_2}$, equivalently, 
\[
q_2=\tfrac{q_1}{1-2q_1}.
\]
Furthermore, we have
\begin{align*}
A_\circ & =( \tfrac{5}{2} q_1+\tfrac{3}{2}q_2+\lambda)^{-1} =\tfrac{q_1}{q_1-1}=-\tfrac{q_2}{q_2+1},
\\
\alpha_\circ & = \begin{bmatrix} 1 & 0 \\ 0 & -1\end{bmatrix} +
\begin{bmatrix}1- q_1 \\ 1+ q_2\end{bmatrix}
( \tfrac{5}{2} q_1+\tfrac{3}{2}q_2+\lambda)^{-1}\begin{bmatrix} \tfrac{5}{2} & \tfrac{3}{2}\end{bmatrix}\\
& %= \begin{bmatrix} \tfrac{5}{2} & \tfrac{3}{2}\end{bmatrix}
\begin{bmatrix} 1 & 0 \\ 0 & -1\end{bmatrix}
-\begin{bmatrix} q_1 \\ q_2 \end{bmatrix}\begin{bmatrix} \tfrac{5}{2} & \tfrac{3}{2}\end{bmatrix}=\begin{bmatrix}
1-\tfrac{5}{2}q_1 & -\tfrac{3}{2} q_1 \\
-\tfrac{5}{2}q_2 &-1 -\tfrac{3}{2} q_2
\end{bmatrix}.
\end{align*}
Solving for $\lambda$ completely in terms of $q_1$ we have
\[
\lambda=\tfrac{-1+q_1}{q_1}-\tfrac{5}{2}q_1-\tfrac{3}{2}\tfrac{q_1}{1-2q_1}.
\]
Set $z=\tfrac{-1+q_1}{q_1}$, or equivalently, $q_1=\tfrac{-1}{z-1}$. Then one checks that
$\tfrac{q_1}{1-2q_1}=-\tfrac{1}{z+1}$. It follows that
\[
\lambda= z+\tfrac{5}{2}\cdot \tfrac{1}{z-1}+\tfrac{3}{2}\cdot \tfrac{1}{z+1}=\omega(z).
\]
Rewriting $A_\circ$ in terms of $z$ we have $A_\circ=\tfrac{q_1}{q_1-1}=\tfrac{1}{z}$. Stability of $A_\circ$ is therefore equivalent to $|z| >1$. Also rewriting $\alpha_\circ$ in terms of $z$ we have
\[
\alpha_\circ=\alpha_\circ(z)=\begin{bmatrix}
1+\tfrac{5}{2}\cdot \tfrac{1}{z-1} & \tfrac{3}{2}\cdot \tfrac{1}{z-1} \\
\tfrac{5}{2}\cdot \tfrac{1}{z+1} & -1+ \tfrac{3}{2}\cdot \tfrac{1}{z+1} 
\end{bmatrix}.
\]
We see that the resolvent set of $T_\omega$ is the set of all $\lambda=\omega(z)$ for which
$|z| >1$ and $\alpha_\circ(z)$ is stable. Moreover, for $|z|\leq 1$ the matrix $A_\circ$ is not stable, so for those values of $z$ we have that $\lambda=\omega(z)\in\sigma(T_\omega)$. Notice that the map from $z$ to $\lambda$ is not one-to-one, so for a given $\lambda$ there may be a value of $z$ with $|z|>1$ for which $\lambda=\omega(z)$ such that $\alpha_\circ(z)$ is stable, while there may be another value of $z$ with $|z|>1$ for which $\lambda=\omega(z)$ such that $\alpha_\circ(z)$ is not stable. 

Our next remark is that the curve $\omega(\mathbb{T})$ gives the essential spectrum (see \cite{GtHJR19a}), and in each connected component of the complement of this curve the Fredholm index of the operator $T_\omega -\lambda I$ is independent of $\lambda$. In addition, the resolvent of $T_\omega$ is a union of connected components of the complement of the curve $\omega(\mathbb{T})$. Hence, we have to check invertibility of $T_\omega -\lambda I$ for just one point $\lambda$ in each component. See  Figure 2 for the connected components.

\begin{figure}
\includegraphics[height=5cm]{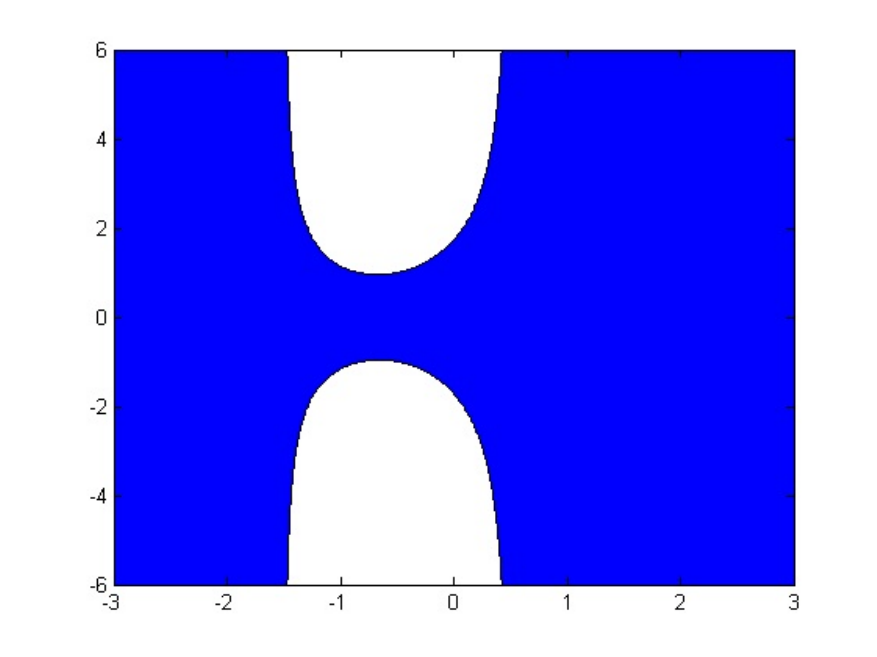}
\caption{Spectrum of $T_\omega$, where $\om(z)=\frac{z^3+3z+1}{z^2-1}$}
\end{figure}

Taking $z$ real we have that $\lambda=\omega(z)$ is real, and the whole real line is in one connected component. Taking $z=0$ gives $\lambda=-1$. For $z=0$ the matrix $A_\circ$ is not stable, so $\lambda=-1$ is in the spectrum of $T_\omega$, and hence the whole connected component containing the real line is in the spectrum of $T_\omega$. Taking $z=3i$ we have that $\lambda=-0.1+1.8i$ which is in the connected component not containing the real line and above the real line. For this value of $z$ we find $A_\circ$ stable (as $|z|>1$) and 
\[
\alpha_\circ(3i) =\begin{bmatrix} 
\tfrac{3}{4} -\tfrac{3}{4} i & -\tfrac{3}{20} - \tfrac{9}{20} i \\[1mm]
\tfrac{1}{4}  - \tfrac{3}{4} i & -\tfrac{17}{20} - \tfrac{9}{20} i 
\end{bmatrix},
\]
which has spectral radius $0.9572$ and hence is stable. Thus, this component belongs to the resolvent of $T_\omega$. Taking $z=-3i$, we obtain that $\lambda =-0.1-1.8i$, which is in the connected component not containing the real line and below the real line. We have
 $\alpha_\circ(-3i)=\overline{\alpha_\circ(3i)}$, and hence also $\alpha_\circ(-3i)$ is stable. This implies that the connected component containing the corresponding value of $\lambda$ is in the resolvent of $T_\omega$ as well. Hence the spectrum of $T_\omega$ is the area in blue in Figure 2, while the resolvent set of $T_\omega$ is the area in white in Figure 2.
\end{ourexample}

\begin{ourexample}
Let 
\[
\Omega(z)=\begin{bmatrix} \tfrac{1}{z-1} & \tfrac{z}{z+1} \\ \tfrac{z}{z-1} & \tfrac{1}{z+1}\end{bmatrix}=\begin{bmatrix} 0 & 1 \\ 1 & 0 \end{bmatrix} +
\begin{bmatrix} 1 & -1 \\ 1 & 1 \end{bmatrix} 
\begin{bmatrix} \tfrac{1}{z-1} & 0 \\ 0 & \tfrac{1}{z+1}\end{bmatrix}.
\]
Since there are no poles in $\BC\setminus \overline{\BD}$, we see that a minimal realization of the form \eqref{Omreal} with $A$ stable and $\al$ semi-stable, has $A=0$ on a zero-dimensional space, $B=0$ ($0\times 2$), $C=0$ ($2\times 0$), while the other matrices can be taken as 
\[
R_0=\begin{bmatrix} 0 & 1 \\ 1 & 0 \end{bmatrix},\quad \gamma=\begin{bmatrix} 1 & -1 \\ 1 & 1 \end{bmatrix},\quad \alpha=\begin{bmatrix} 1 & 0 \\ 0 & -1\end{bmatrix},\quad \beta=\mat{1&0\\0&1}=I_2.
\] 
This example fits in the context of Corollary \ref{C:polesindisc}. Hence the algebraic Riccati equation \eqref{eq:Ricc} has $Q=0$ as a solution, interpreted as a $2\times 0$ matrix. The matrix $A_\circ$ is the zero matrix on a zero-dimensional space, the matrix $\alpha_\circ$ is given as a function of $\lambda$ by
\begin{align*}
\alpha_\circ&=\al_\circ(\la)= \alpha-\beta(R_0-\lambda I)^{-1}\gamma 
=\begin{bmatrix} 1 & 0 \\ 0 & -1\end{bmatrix}-\begin{bmatrix} -\lambda & 1 \\ 1 & -\lambda\end{bmatrix}^{-1} \begin{bmatrix} 1 & -1 \\ 1 & 1 \end{bmatrix} \\
&=
\begin{bmatrix} 1 & 0 \\ 0 & -1\end{bmatrix}+\begin{bmatrix}
\tfrac{1}{\lambda-1} & -\tfrac{1}{\lambda+1} \\
\tfrac{1}{\lambda-1} & \tfrac{1}{\lambda+1}\end{bmatrix} 
= \begin{bmatrix}
\tfrac{\lambda}{\lambda-1} & -\tfrac{1}{\lambda+1} \\
\tfrac{1}{\lambda-1} & -\tfrac{\lambda}{\lambda+1}\end{bmatrix}, \mbox{ for $\la\neq \pm 1$}.
\end{align*}
Set $z=\tfrac{1}{\lambda}$, then one readily checks that $\alpha_\circ(\lambda)=-\Omega(z)=-\Omega(\tfrac{1}{\lambda})$. It follows that $T_\Omega -\lambda I$ is invertible if and only if $-\Omega(\tfrac{1}{\lambda})$ is a stable matrix. However, since $\det \Omega(z)=-1$ for all $z$ the matrix $\Omega(z)$ cannot be stable for any $z\in \mathbb{C}$. It follows that the resolvent set is empty and hence $\sigma(T_\Omega)=\mathbb{C}$. 

Next, we determine the essential spectrum of $T_\Omega$. The matrix pencil $L(\lambda, z)$ is given by 
\[
L(\lambda,z)=\begin{bmatrix} \alpha -z I & \beta \\ \gamma & R_0-\lambda\end{bmatrix}
= \begin{bmatrix}
1-z & 0 & 1 & 0 \\ 0 & -1 -z & 0 & 1 \\ 1 & -1 & -\lambda & 1 \\ 1 & 1 & 1 & -\lambda
\end{bmatrix}.
\]
We obtain 
\[
\det L(\lambda, z)= \lambda^2 (z^2-1) -(z^2-1) -2\lambda z = (z^2-1)(\lambda^2-1)-2\lambda z.
\]
Note that for each $\la$, the polynomial $\det L(\lambda, z)$ is never identically zero. Hence $E(\Om)=\emptyset$.
%
%Secondly, we show that the elements of $\sigma(R_0)=\{1, -1\}$ are not in the essential spectrum of $T_\Omega$. To this end, consider
%\[
%(\Omega(z)-I)^{-1} = \begin{bmatrix} \tfrac{1}{2}(z-1) & \tfrac{1}{2}(z-1) \\
%\tfrac{1}{2}(z+1) & \tfrac{(z+1)(z-2)}{2z}\end{bmatrix},
%\]
%and
%\[
%(\Omega(z)+I)^{-1} = \begin{bmatrix}  \tfrac{(z-1)(z+2)}{2z}
%  & -\tfrac{1}{2}(z-1) \\
%-\tfrac{1}{2}(z+1) & \tfrac{1}{2}(z+1)\end{bmatrix},
%\]
%neither of which have a pole on the unit circle. It follows that $\Om(z)-I$ and $\Om(z)+I$ both have no zero on the unit circle, and hence $T_{\Om} \pm I$ are both Fredholm. 
%
According to Theorem \ref{thm:essspec}, $\lambda$ is in the essential spectrum if and only if it is an eigenvalue of $\Omega(z)$ for some $z\neq \pm1$ on the unit circle or $\det L(\la,\pm1)=0$. Our formula for $\det L(\la,z)$ shows that the latter occurs only when $\la=0$. 
Since the determinant of $\Omega(z)$ is $-1$ and the trace of $\Omega(z)$ is $\tfrac{2z}{z^2-1}$, the characteristic polynomial of $\Omega(z)$ is given by
\[
p_{\Om(z)}(\la)=\lambda^2-\tfrac{2z}{z^2-1}\lambda -1 =(\lambda -\tfrac{z}{z^2-1})^2-1 -\tfrac{z^2}{(z^2-1)^2}.
\]
Hence the characteristic equation of $\Omega(z)$ becomes
\[
(\lambda -\tfrac{z}{z^2-1})^2 =1+\tfrac{z^2}{(z^2-1)^2 }=\tfrac{z^4-z^2+1}{z^4-2z^2+1}.
\]
Specializing to $z=e^{i\theta}$ gives
\[
(\lambda - \tfrac{e^{i\theta}}{e^{i2\theta}-1} )^2= \tfrac{e^{i4\theta}-e^{i2\theta}+1}{e^{i4\theta}-2e^{i2\theta}+1} =\tfrac{e^{i2\theta}+e^{-i2\theta}-1}{e^{i2\theta}+e^{-i2\theta}-2 } =
\tfrac{1-2\cos(2\theta)}{2-2\cos(2\theta)} =\tfrac{3-4\cos^2(\theta)}{4\sin^2(\theta)}.
\]
We obtain that for $\theta$ such that $3-4\cos^2(\theta) \geq 0$ we have
\[
\lambda =\tfrac{e^{i\theta}}{e^{i2\theta}-1} \pm \tfrac{\sqrt{3-4\cos^2(\theta)}}{2|\sin(\theta)|}=
\tfrac{1}{e^{i\theta}-e^{-i\theta}}  \pm \tfrac{\sqrt{3-4\cos^2(\theta)}}{2|\sin(\theta)|}
= -\tfrac{i}{2\sin(\theta)} \pm \tfrac{\sqrt{3-4\cos^2(\theta)}}{2|\sin(\theta)|}.
\]
Observe that these values of $\lambda$ are on the unit circle. Noticing that $3-4\cos^2(\theta) \geq 0$ precisely when $\tfrac{\pi}{6}\leq \theta \leq \tfrac{5\theta}{6}$ or $\tfrac{7\pi}{6}\leq \theta \leq \tfrac{11\theta}{6}$ we have that for these values of $\theta$ the two values of $\lambda$ trace out two curves which lie on the unit circle and connect respectively $e^{i\pi/6}$ to $e^{i5\pi/6}$ and $e^{i7\pi/6}$ to $e^{i11\pi/6}$.

For the case where $3-4\cos^2(\theta) <0$ the values of $\lambda$ are given by
\[
\lambda=\left(-\tfrac{1}{2\sin(\theta)} \pm \tfrac{\sqrt{-3+4\cos^2(\theta)}}{2|\sin(\theta)|}\right) \cdot i .
\]
The values of $\theta$ for which $3-4\cos(\theta)^2 <0$ are given by $-\tfrac{\pi}{6} <\theta <\tfrac{\pi}{6}$ or $\tfrac{5\pi}{6}<\theta <\tfrac{7\pi}{6}$. 
One easily checks that for these values of $\theta$ the two values of $\lambda$ trace out the whole imaginary line, excluding $0$ since $\la=0$ would occur at $\theta=0$ or $\theta=\pi$ in which case $\sin(\theta)=0$. However, $\la=0$ is in the essential spectrum because that is when $\det L(\la,\pm1)=0$.

It follows that the essential spectrum of $T_\Omega$ is the union of the imaginary line with two arcs on the unit circle. See Figure 3.

\begin{figure}
\includegraphics[height=4cm]{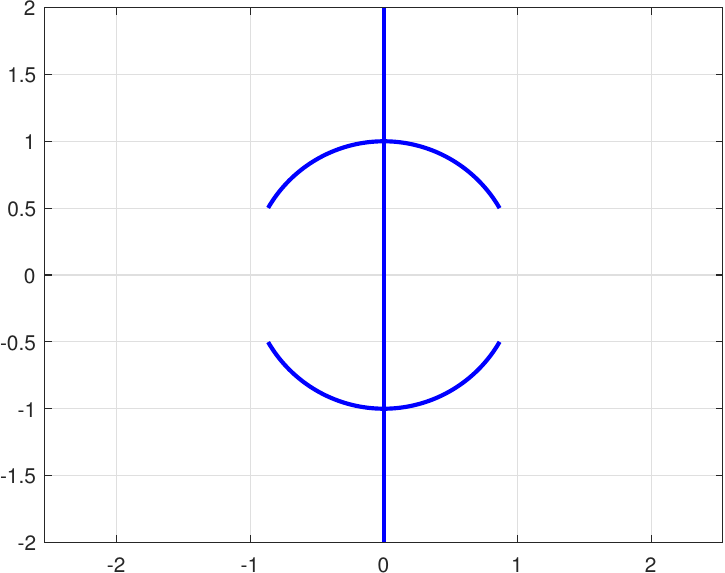}
\caption{The essential spectrum of $T_\Omega$ for $\Omega(z)=\begin{bmatrix} \tfrac{1}{z-1} & \tfrac{z}{z+1} \\ \tfrac{z}{z-1} & \tfrac{1}{z+1}\end{bmatrix}$.}
\end{figure}
\end{ourexample}

\begin{ourexample}
We return to Example \ref{Ex1}. In that case we have
\[
\Omega(z)=\begin{bmatrix} 2 & \tfrac{1}{z-1} \\ 0 & 2 \end{bmatrix} = \begin{bmatrix} 2  & 0 \\ 0 & 2\end{bmatrix} + \begin{bmatrix} 1 \\ 0 \end{bmatrix} (z-1)^{-1} \begin{bmatrix} 0 & 1\end{bmatrix}.
\]
Observe that the right hand side of the above equation constitutes a minimal realization of $\Omega$.  We saw before that $E(\Omega)=\{2\}$. Further,
\[
L(\lambda, z)=\begin{bmatrix} 1-z & 0 & 1 \\ 1 & 2-\lambda & 0 \\ 0 & 0 & 2-\lambda\end{bmatrix},
\]
so that $\det L(\lambda,z)=(2-\lambda)^2 (1-z)$. It follows that all $\lambda\not=2$ are in the essential spectrum, and since the essential spectrum is closed, we have $\sigma_\tu{ess} (T_\Omega)=\mathbb{C}$. That $2\in \sigma_\tu{ess} (T_\Omega)=\mathbb{C}$ also follows since $E(\Om)\subset \sigma_\tu{ess} (T_\Omega)$, by Lemma \ref{L:EOm}.
\end{ourexample}

\paragraph{\bf Acknowledgements}
This work is based on research supported in part by the National Research Foundation of South Africa (NRF, Grant Numbers 118513, 127364 and 145688) and the DSI-NRF Centre of Excellence in Mathematical and Statistical Sciences (CoE-MaSS). Any opinion, finding and conclusion or recommendation expressed in this material is that of the authors and the NRF and CoE-MaSS do not accept any liability in this regard.

\end{document}